\documentclass[11pt]{amsart}

\usepackage{amsmath,amssymb,amsthm,mathrsfs,enumitem,url,color,accents}
\usepackage[bookmarks=true, bookmarksnumbered=true]{hyperref}
\usepackage[hmargin=2cm,vmargin=3cm]{geometry}

\setlist[enumerate]{leftmargin=*}
\setlist[itemize]{labelindent=\parindent, leftmargin=*}

\numberwithin{equation}{section}

\theoremstyle{plain}
\newtheorem{thm}{Theorem}[section]
\newtheorem{lem}[thm]{Lemma}
\newtheorem{prop}[thm]{Proposition}
\newtheorem{cor}[thm]{Corollary}
\theoremstyle{definition}
\newtheorem{defn}[thm]{Definition}
\theoremstyle{remark}
\newtheorem{rem}[thm]{Remark}
\newtheorem{conj}[thm]{Conjecture}
\newtheorem{hypo}[thm]{Hypothesis}

\newcommand\Hom{\operatorname{Hom}}

\newcommand\Irr{\operatorname{Irr}}

\newcommand\GL{\mathrm{GL}}
\newcommand\Mp{\mathrm{Mp}}
\newcommand\OO{\mathrm{O}}
\newcommand\PGL{\mathrm{PGL}}
\newcommand\SL{\mathrm{SL}}
\newcommand\SO{\mathrm{SO}}
\newcommand\Sp{\mathrm{Sp}}

\newcommand\A{\mathbb{A}}
\newcommand\C{\mathbb{C}}

\newcommand\R{\mathbb{R}}

\newcommand\Z{\mathbb{Z}}

\newcommand{\BIGOP}[1]{\mathop{\mathchoice%
{\raise-0.22em\hbox{\huge $#1$}}%
{\raise-0.05em\hbox{\Large $#1$}}{\hbox{\large $#1$}}{#1}}}

\newcommand{\BIGboxplus}{\mathop{\mathchoice%
{\raise-0.35em\hbox{\huge $\boxplus$}}%
{\raise-0.15em\hbox{\Large $\boxplus$}}{\hbox{\large $\boxtimes$}}{\boxtimes}}}

\title{The Shimura--Waldspurger correspondence for $\Mp(2n)$}
\author{Wee Teck Gan}
\address{Department of Mathematics, National University of Singapore, 10 Lower Kent Ridge Road, Singapore 119076}
\email{matgwt@nus.edu.sg}
\author{Wen-Wei Li}
\address{Academy of Mathematics and Systems Science, 55 Zhongguancun E.Rd, Beijing, P.\! R.\! China 100190}
\address{University of Chinese Academy of Sciences, 19A Yuquan Rd, Beijing, P.\! R.\! China 100049}
\email{wwli@math.ac.cn}
\date{\today}
 
\begin{document}

\begin{abstract}
We describe some recent developments and formulate some conjectures in the genuine representation theory and the study of automorphic forms of the metaplectic group $\Mp(2n)$, from the point of view of the theta correspondence as well as from the point of view of the theory of endoscopy and the trace formula. 
\end{abstract}

\maketitle

\section{Introduction}
In a seminal 1973 paper \cite{shimura}, Shimura revolutionized the study of half integral weight modular forms by establishing a lifting
\[
\begin{array}{c}
 \{ \text{Hecke eigenforms of weight $k + \tfrac{1}{2}$ and level $\Gamma_0(4)$} \} \\
 \downarrow \\
 \{ \text{Hecke eigenforms of weight $2k$ and level $\SL_2(\Z)$} \}
\end{array} 
\]
for $k \in \mathbb{N}$.
Subsequently, Niwa \cite{niwa} and Shintani \cite{shintani} explicitly constructed the Shimura lifting and its inverse by using theta series lifting.
Then, in two influential papers \cite{w1,w2}, Waldspurger studied this construction in the framework of automorphic representations of the metaplectic group $\Mp(2)$, which is a nonlinear two-fold cover of $\SL_2 = \Sp_2$.
Namely, he described the automorphic discrete spectrum of $\Mp(2)$ precisely in terms of that of $\PGL(2) = \SO(3)$ via the global theta lifts between $\Mp(2)$ and (inner forms of) $\SO(3)$.
Subordinate to this global result is the local Shimura correspondence, which is a classification of irreducible genuine representations of $\Mp(2)$ in terms of that of $\SO(3)$ and was also established by Waldspurger. For an expository account of Waldspurger's result taking advantage of 30 years of hindsight and machinery, the reader can consult \cite{g3}.

\vskip 10pt

In this expository paper, which is a write-up of the talks given by the two authors in the Simons Symposium at Elmau, April 2016,
we present a similar conjectural description of the automorphic discrete spectrum of $\Mp(2n)$, which is a nonlinear two-fold cover of $\Sp(2n)$, and describe recent progress towards this conjecture.  This conjectural description of the automorphic discrete spectrum is in the style of Arthur's conjecture for the automorphic discrete spectrum of connected linear reductive groups and can be viewed as a description of the automorphic discrete spectrum of $\Mp(2n)$ in terms of that for $\SO(2n+1)$. There are two natural approaches one might take to establish this Arthur conjecture for $\Mp(2n)$:
\vskip 10pt

\begin{itemize}
\item Motivated by Waldspurger's work for $\Mp(2)$, it is natural to attempt to use the global theta lifts between $\Mp(2n)$ and (inner forms of) $\SO(2n+1)$ to relate these automorphic discrete spectra.  However, we encounter a well-known difficulty.
For any irreducible cuspidal automorphic representation $\pi$ of $\Mp(2n)$, there is an obstruction for the nonvanishing of its global theta lift to $\SO(2n+1)$ given by the vanishing of the central $L$-value $L(\frac{1}{2}, \pi)$.
Thus, if we would follow Waldspurger's approach, then we would need the nonvanishing of the central $L$-value $L(\frac{1}{2}, \pi, \chi)$ twisted by some quadratic Hecke character $\chi$.
The existence of such $\chi$ is supplied by Waldspurger \cite{w2} in the case of $\Mp(2)$ (or equivalently $\PGL(2)$) as a consequence of the nonvanishing of a global theta lift, and a completely different proof and extension to the case of $\GL(2)$ is given by Friedberg--Hoffstein \cite{fh}.
However, in the higher rank case, this seems to be a very difficult problem in analytic number theory.
 \vskip 5pt
 
 In the first part of the paper, based on the talk by the first author,  we  explain how one can overcome this difficulty by considering instead theta liftings between $\Mp(2n)$ and $\SO(2r+1)$ with $r$ much larger than $n$. 
 In particular, we will  sketch  a recent proof by the first author and A. Ichino \cite{GI} of the tempered part of this conjecture: this is the analog of Waldspurger's theorem for $\Mp(2n)$. 
 \vskip 10pt
 
 \item Motivated by Arthur's classification of the automorphic discrete spectra of classical groups,  one might attempt to develop a trace formula comparison between $\Mp(2n)$ and $\SO(2n+1)$ which is founded on a theory of endoscopy and local character identities. For the case of $\Mp(2)$, the local aspects of this approach was done in the PhD thesis of Jason Schulze. For $\Mp(2n, \R)$, such a theory of endoscopy was pioneered by the work of J. Adams \cite{adams} and D. Renard \cite{renard}.  These culminated in the PhD and subsequent work of the second   author \cite{wwli1, wwli2, wwli3, wwli4}, which develops a full theory of endoscopy   and lays the groundwork for the stabilisation of the invariant trace formula for for $\Mp(2n)$. 
\vskip 5pt
   
 In the second part of this paper, based on the talk by the second author,  we give a sketch of this theory of endoscopy and describe some recent progress of the second author towards the stable trace formula  for $\Mp(2n)$. In particular, we will formulate  the expected local character identities that one expects from the theory of endoscopy for $\Mp(2n)$. In a recent preprint of Caihua Luo \cite{luo2}, this local character identity has been shown for the local L-packets defined by local theta correspondence, thus reconciling the two approaches in question. The proof uses the stabilisation of the elliptic part of the trace formula of $\Mp(2n)$ due to the second author \cite{wwli3}.
 \end{itemize}
  \vskip 10pt

\noindent{\bf Acknowledgments}: We thank the Simons Foundation for its generous travel and local support during the duration of the Simons Symposium. We are also grateful to Caihua Luo for his comments on an earlier draft.
\vskip 10pt

\section{\bf Automorphic Discrete Spectrum via Theta Correspondence}

\subsection{\bf Local Shimura correspondence.}
Let $k$ be a local field of characteristic zero and fix a nontrivial character
\[  \psi:  k  \longrightarrow \C^{\times}. \]
Let $ \tilde{{\rm Irr}}(\Mp(2n,k))$ be the set of isomorphism classes of 
irreducible genuine representations $\Mp(2n,k)$. We have the following theorem,  which was shown by Adams--Barbasch \cite{ab1,ab2} in the archimedean case and by Gan--Savin \cite{gs} in the nonarchimedean case:

\vskip 5pt

\begin{thm}  \label{T:2.1}
There is a bijection, depending on $\psi$:
\[   \theta_{\psi}:  \tilde{{\rm Irr}}(\Mp(2n,k))   \longleftrightarrow \bigcup_{V_n} {\rm Irr} (\SO(V_n)) \]
as $V_n$ runs over all $2n+1$-dimensional quadratic space of discriminant 1.
\end{thm}
\vskip 5pt

The bijection is defined by local theta correspondence with respect to $\psi$.
By combining with the local Langlands correspondence (LLC) for $\SO(2n+1)$ (due to Arthur \cite{Art1} and  Moeglin \cite{m2,m3}), one gets:
\vskip 5pt

\begin{cor}  \label{C:2.2}
There is a bijection depending on $\psi$:
\[  \tilde{\rm Irr} (\Mp(2n,k)) \longleftrightarrow    \{  (\phi, \eta)  \}  \]
where
\begin{itemize}
\item $\phi:  WD_k  \longrightarrow  \Sp_{2n}(\C)$ is an L-parameter for $\SO(2n+1)$; 
\vskip 10pt

\item $\eta \in {\rm Irr} (S_{\phi})$, where $S_{\phi}$ is the component group of $\phi$.
\end{itemize}
\end{cor}
\vskip 10pt

This corollary is a LLC for $\Mp(2n,k)$. Thus, given $\phi$, have local L-packet
\[  \tilde{\Pi}_{\phi, \psi}  =  \{   \sigma_{\phi,  \eta}:  \eta  \in  {\rm Irr} (S_{\phi}) \}. \]
The paper \cite{gs} also established several desiderata of this LLC for $\Mp(2n)$, some of which are recalled in \S 3 later on.
 
\vskip 10pt

\subsection{\bf Global Results.}
Now let $F$ be a number field with ad\`{e}le ring $\A$, and let $\psi = \prod_v \psi_v : F \backslash \A \longrightarrow \C^{\times}$ be a nontrivial character. We shall now discuss some results of Ichino and the first author \cite{GI}.
\vskip 15pt

\begin{thm} \label{T:A}
There is a decomposition 
\[  \tilde{\mathcal{A}} = \bigoplus_{\Psi}  \tilde{\mathcal{A}}_{\Psi, \psi} \]
where $\Psi$ runs over elliptic global A-parameters of $\Mp(2n)$ and each $\tilde{\mathcal{A}}_{\Psi}$ is a near equivalence class determined by $\Psi$ and $\psi$.
\end{thm} 
\vskip 10pt

Let us explain the various terms and notions in the theorem.
Elliptic global A-parameters of $\Mp(2n)$ are simply  elliptic global A-parameters of $\SO(2n+1)$, i.e. they have the form
  \[  \Psi  =\BIGboxplus_i  \Pi_i \boxtimes  S_{d_i} \] 
 where:
 \vskip 5pt

 \begin{itemize}
 \item $\Pi_i$ is a cuspidal representation of $\GL(n_i)$ which is of symplectic (resp. orthogonal) type if $d_i$ is odd (resp. even);
 \vskip 10pt
  
 \item $S_{d_i}$ is the irreducible  representation of $\SL_2(\C)$ of dimension $d_i$;
 \vskip 10pt

 \item $n_1d_1 +...+n_rd_r  = 2n$;

 \vskip 10pt

 \item $\Pi_i \boxtimes S_{d_i}  \not\simeq  \Pi_j \boxtimes S_{d_j}$  if $i \ne j$.
  \end{itemize}
For such a global A-parameter $\Psi$, we define its global component group by
\[  S_{\Psi} =  \prod_i  \Z/2\Z \cdot a_i,  \]
 i.e. $S_{\Psi}$ is a $\Z/2\Z$-vector space with a canonical basis indexed by the ``irreducible summand" $\Pi_i \boxtimes S_{d_i}$ of $\Psi$.
  \vskip 5pt

 Each $\Psi$ gives rise to a near equivalence class of representations of $\Mp(2n,\A)$ via the local Shimura correspondence, as follows.
  For almost all $v$, $\Pi_{i,v}$ is unramified for all $i$, with L-parameter $\phi_{i,v}$.  Set
 \[   \phi_{\Psi_v} =  \bigoplus_i  \phi_{i,v}  \otimes \left( 
  \begin{array} {cccc} 
 |-|_v^{(d_i -1)/2}  &     &   & \\
 
 & |-|_v^{(d_i -3)/2}  &  &  \\
 &  &   \ddots &  \\
 &  &  &  |-|_v^{-(d_i -1)/2}  \end{array}  \right).  \]
  This defines an unramified L-parameter for $\SO(2n+1, F_v)$ and determines an unramified representation of $\Mp(2n, F_v)$ by local Shimura correspondence determined by $\psi_v$. This collection of unramified representations determines a near equivalence class for $\Mp(2n)$ which we call the near equivalence class associated to the global A-parameter  $\Psi$ and $\psi$. 
 \vskip 10pt
 
 \subsection{The Tempered Part.}
 We say that an $A$-parameter $\Psi = \boxplus_i  \Pi_i \boxtimes  S_{d_i}$ is tempered if $d_i=1$ for all $i$.  For such tempered A-parameters, we have:
  \vskip 5pt
  
\begin{thm} \label{T:B}
We have an explicit description of $\tilde{\mathcal{A}}_{\Psi, \psi}$ when $\Psi = \boxplus_i  \Pi_i$ is tempered.
\end{thm}
\vskip 5pt
 
 Let us give the explicit description claimed in the above theorem. 
More precisely, one has:
\vskip 5pt

\begin{itemize}
\item for each $v$, $\Psi_v $ gives rise to the local L-packet $\tilde{\Pi}_{\Psi_v, \psi_v}$
\vskip 10pt

\item the global L-packet  $\tilde{\Pi}_{\Psi,\psi}  = \otimes_v  \tilde{\Pi}_{\Psi_v, \psi_v}$.
\vskip 10pt

\item global and local component groups with a natural diagonal map:
\[ \Delta:  S_{\Psi}  = \bigoplus_i \Z/2\Z a_i  \longrightarrow  S_{\Psi, \A} :=  \prod_v  S_{\Psi_v}. \]

\item a quadratic character
\[  \tilde{\epsilon}_{\Psi} :  S_{\Psi}  \longrightarrow \{  \pm 1 \} \]
given by:
 \[  \tilde{\epsilon}_{\Psi}(a_i) =  \epsilon(1/2, \Pi_i). \] 
 \end{itemize}
 
 \subsection{Multiplicity Formula.}
 Now each $\eta  = \boxtimes_v  \eta_v \in {\rm Irr}(S_{\Psi, \A})$ gives rise to 
\[  \sigma_{\eta} = \otimes_v  \sigma_{\eta_v}  \in \tilde{\rm Irr}(\Mp(2n,\A)), \]
which is an abstract representation of $\Mp(2n,\A)$.
One has 
\[  \tilde{\mathcal{A}}_{\Psi,\psi}  =  \bigoplus_{\eta  \in {\rm Irr} S_{\Psi,\A}} m_{\eta}  \cdot \sigma_{\eta}  \]
with
\[  m_{\eta}  = \begin{cases}
1 & \text{  if $\Delta^*(\eta)  =  \tilde{\epsilon}_{\Psi}$; } \\
0 &  \text{ otherwise.} 
\end{cases} \]
  When $n=1$, this is precisely the result of Waldspurger from \cite{w1, w2}. For general $n$, it was conjectured in \cite{ggp}.
  \vskip 10pt
  
  \subsection{Theta Lifting}
 In the rest of this section, we are going to give a sketch of the proof of Theorem \ref{T:A} and Theorem \ref{T:B}.
The  proof of these theorems relies on the global theta correspondence between $\Mp(2n)$ and the split $\SO(2n+2r +1)$ with $r > n$.
 This theta correspondence has been studied by J.S. Li \cite{li1, li2, li3} both locally and globally and we shall formulate his results shortly.
 \vskip 10pt

 Given an abstract representation
 \[  \sigma = \otimes_v \sigma_v  \in \tilde{{\rm Irr}}(\Mp(2n,\A)),  \]
 set
 \[  m_{disc}(\sigma) = \dim \Hom_{\Mp(2n)} ( \sigma, \tilde{\mathcal{A}}_{disc})  \] 
 \[  m(\sigma)  =  \dim \Hom_{\Mp(2n)} ( \sigma, \tilde{\mathcal{A}}),  \]
 where  $\tilde{\mathcal{A}}_{disc}$ stands for the automorphic discrete spectrum whereas $\tilde{\mathcal{A}}$ stands for the space of all automorphic forms.
 \vskip 10pt
 
 Because $r > n$, so that we are in the so-called stable range, 
 we have a nonzero local theta lift 
 \[ 0 \ne  \theta_{\psi_v}(\sigma_v)  \in {\rm Irr} \left( \SO(2n+2r+1, F_v) \right). \]
 Set
 \[   \theta^{abs}_{\psi} (\sigma)  =  \otimes_v   \theta_{\psi_v}(\sigma_v)  \in {\rm Irr} \SO(2n+2r+1,\A). \]
 Likewise, we have the multiplicity $m_{disc}( \theta^{abs}_{\psi} (\sigma))$ and $m( \theta^{abs}_{\psi} (\sigma))$.
  \vskip 10pt
  
  \subsection{Results of J.S. Li} We have \cite{li3}:
 \begin{thm}[J.S. Li]  \label{T:li}
 Given $ \sigma  \in \tilde{{\rm Irr}}(\Mp(2n,\A))$, 
 one has:
 \[  m_{disc}(\sigma)  \leq m_{disc}(\theta_{\psi}^{abs}(\sigma)) \leq m(\theta^{abs}_{\psi}(\sigma))  \leq m(\sigma). \]
  \end{thm}
 
 \vskip 5pt
 
 \begin{cor}
Let $\Sigma \subset \tilde{\mathcal{A}}_{disc}$ be a near equivalence class, say
\[  \Sigma = \bigoplus_i m_i  \sigma_i.  \] 
 Set 
 \[  \theta^{abs}_{\psi}(\Sigma)  = \bigoplus_i  m_i  \theta^{abs}_{\psi}(\sigma_i). \]
 Then
 \[  \theta^{abs}_{\psi}(\Sigma)  \subset \mathcal{A}_{disc}(\SO(2n+2r+1)) \]
 
 \end{cor}

 \subsection{Assignment of A-parameters}
By the above corollary, to $\theta^{abs}_{\psi}(\Sigma)$, one can assign  by Arthur an A-parameter 
\[  \Psi^r = \BIGboxplus_i  \Pi_i \boxtimes S_{d_i} \]
 of $\SO(2n+2r+1)$. We now note:
\vskip 5pt

\begin{prop}  \label{P:psi}
\[  \Psi^r  = \Psi  \boxplus   S_{2r} \]
for an  elliptic A-parameter $\Psi$ of $\SO(2n+1)$. 
\end{prop}
\vskip 5pt
 
Hence, one defines $\Psi$ to be the A-parameter of $\Sigma$, so that
\[  \Sigma =  \tilde{\mathcal{A}}_{\Psi}. \]  
This gives Theorem \ref{T:A}.
 \vskip 5pt
 
 \subsection{Proof of Prop. \ref{P:psi}.} Let us give  a brief sketch of the proof of Proposition \ref{P:psi}.
Consider in two different ways the partial standard L-function:
\[  L^S(s,  \theta^{abs}_{\psi}(\sigma))   \]
On one hand, it  follows by our understanding of the local theta correspondence for unramified representations that this partial L-function  is equal to
\[    L^S_{\psi}(s,  \sigma)  \cdot \zeta^S \left(s+ r-\frac{1}{2}\right) \cdot  \zeta^S \left(s+ r - \frac{3}{2}\right)  \cdot\cdots\cdot  \zeta^S \left(s - r + \frac{1}{2}\right). \]
This has largest pole at
\[  s  = r + \frac{1}{2} > n +1. \]
 
On the other hand, from the form of the A-parameter of $ \theta^{abs}_{\psi}(\sigma)$,  the same partial L-function is equal to
\[  \prod_i  \prod_{j=1}^{d_i}  L^S\left( s+ \frac{d_i +1}{2} - j, \Pi_i \right). \]
For this to have largest pole at $s = r+1/2$, one sees readily that it is necessary for  $ S_{2r} \subset \Psi^r$. 
 \vskip 10pt
 
 \subsection{Equality?}
We have shown that there is a containment
\[ \theta^{abs}_{\psi}(\Sigma)  =  \theta^{abs}_{\psi}( \tilde{\mathcal{A}}_{\Psi,\psi})  \subset \mathcal{A}_{\Psi^r}  \]
and it is natural to ask if equality holds.   If equality holds, then one can simply transport the description of $\mathcal{A}_{\Psi^r}$ to get 
a description of $\tilde{\mathcal{A}}_{\Psi,\psi}$ in the style of Theorem \ref{T:B}. To this end, we first have:
 
\vskip 5pt

\begin{lem} 
Given $\pi \subset  \mathcal{A}_{\Psi^r}$, one has 
\[  \pi \cong \theta^{abs}_{\psi}(\sigma)  \]
as abstract representations, for some irreducible genuine representation $\sigma$ of $\Mp(2n,\A)$.  
\end{lem}
\vskip 5pt

As a consequence of the lemma, since $m_{disc}(\pi) > 0$, J.S. Li's inequality in Theorem \ref{T:li} implies
\[  m(\sigma)  > 0, \quad \text{i.e. $\sigma$ is automorphic.} \]
 
 \subsection{Key proposition for tempered $\Psi$.} We now come to the key proposition in the proof of Theorem \ref{T:B}.
 \vskip 5pt
 
\begin{prop}  \label{P:key}
If $\Psi$ is tempered, then with $\sigma \in \Pi_{\Psi,\psi}$ as above
\[   m_{cusp}(\sigma) = m_{disc}(\sigma)  =  m(\sigma), \]
with $m_{cusp}(\sigma)$ denoting the multiplicity of $\sigma$ in the cuspidal spectrum. 
In particular,
\[  m_{disc}(\sigma)  = m_{disc}(\theta^{abs}(\sigma)) \]
so that
\[  \theta^{abs}(\tilde{\mathcal{A}}_{\Psi,\psi}) = \mathcal{A}_{\Psi + S_{2r}} . \]
\end{prop}

\vskip 5pt
This implies that the structure of $\tilde{\mathcal{A}}_{\Psi,\psi}$ is ``the same" as that of $\mathcal{A}_{\Psi + S_{2r}}$. For example, one can transport the Arthur multiplicity formula for $\Psi + S_{2r}$  to $\Psi$. In this regard, note that
\[  \epsilon_{\Psi + S_{2r}}  = \tilde{\epsilon}_{\Psi}. \] 
This explains, from the point of view of theta correspondence, why global root numbers appear in the multiplicity formula for tempered A-parameters of $\Mp(2n)$.
 \vskip 10pt
 
 \subsection{ Proof of key proposition.} We give a sketch of the proof of Proposition \ref{P:key}. 
We know
\[  m_{cusp}(\sigma)  \leq m_{disc}(\sigma)  \leq m(\sigma). \]
Hence, we need to show that any $\sigma \hookrightarrow \tilde{\mathcal{A}}$ has image in the cuspidal spectrum $\tilde{\mathcal{A}}_{cusp}$. 
\vskip 10pt
 
Suppose not. Then we note the following:
\vskip 5pt
\begin{itemize}
\item The near equivalence class of  $\sigma$ (given by $\Psi$)  has weak lift to $\GL(2n)$ of the form
\[  \BIGboxplus_i  \Pi_i \]
which is a multiplicity-free sum of cuspidal representations of symplectic type.
\vskip 5pt

\item If $\sigma$ is not not cuspidal, then by a basic theorem of Langlands \cite{langlands1},  
\[  \sigma \subset {\rm Ind}_P^{\Mp(2n)}  \rho, \]
 with $\rho$ cuspidal and contained in the cuspidal support of $\sigma$. If 
\[  M = \tilde{\GL}(k_1) \times_{\mu_2} ......\times_{\mu_2} \tilde{\GL}(k_r)  \times_{\mu_2} \Mp(2n_0), \]
then
\[  \rho  =  \tilde{\tau}_1 \boxtimes ....\boxtimes \tilde{\tau}_r    \boxtimes \sigma_0. \] 
Then $\sigma$ has weak lift to $\GL(2n)$ of the form
\[  \BIGboxplus_i  (\tau_i  \oplus \tau_i^{\vee})  \boxplus \Psi_0  \]
\end{itemize}
\vskip 5pt

 \noindent Since this is not a multiplicity-free sum of symplectic-type cuspidal representations, we obtain  the desired contradiction.

 \subsection{Local packets}
There is only one remaining issue for Theorem \ref{T:B}:  to show that the local packets of $\Mp(2n)$ inherited from $\Psi + S_{2r}$ on $\SO(2n+2r+1)$ is equal to the local L-packet defined by the local Shimura correspondence of Corollary \ref{C:2.2}, i.e. do we have
\[  \tilde{\mathcal{A}}_{\Psi_v, \psi_v} :=  \theta_{\psi_v}(\mathcal{A}_{\Psi^r_v})  = \tilde{\Pi}_{\Psi_v}?  \]
Also, is the labeling by ${\rm Irr} (S_{\Psi_v})$ the same on both sides? 
\vskip 10pt
 
 These are  purely local questions, but is the most difficult and intricate part of  \cite{GI}. The proof uses local results of Moeglin \cite{m2,m3} on explicating the local A-packets of $\SO(2n+1)$ and global arguments involving the global multiplicity formula of Arthur \cite{Art1}  for all inner forms of $\SO(2n+1)$.

\vskip 10pt

\section{\bf Endoscopy and Character Identities for \texorpdfstring{$\Mp(2n)$}{Mp(2n)}}

In this section, we will review the theory of endoscopy for $\Mp(2n)$ developed by the second author. We shall state some conjectures concerning the endoscopic classification of irreducible genuine representations and the automorphic discrete spectrum of $\Mp(2n)$ in the style of Arthur's endoscopic classification for the symplectic and orthogonal groups. 
We note that the second author has stabilised the elliptic part of the invariant trace formula for $\Mp(2n)$. This elliptic stable trace formula will be used to compare the classification provided by the theta correspondence described in \S 2 with the expectations from the endoscopic classification.
\vskip 5pt

 \subsection{Basic formalism}
In order to facilitate the use of the trace formula, we opt to modify the metaplectic group as follows. We  fix a local field $k$ with $\mathrm{char}(k)=0$ and an additive character $\psi$ of $k$.

\begin{itemize}
	\item We enlarge the twofold covering of $\Sp(2n,k)$ to a covering with kernel $\mu_8 \subset \C^\times$, say by pushing-out via $\mu_2 \hookrightarrow \mu_8$. This is natural since Weil's metaplectic covering of $\Sp(2n,k)$ is originally a covering with kernel $\C^\times$ or $U(1)$, and the Schr\"{o}dinger model reduces it to an eightfold covering. The resulting covering group is still denoted by $\Mp(2n,k)$. This procedure leaves the genuine dual $\widetilde{\Irr}(\Mp(2n,k))$ intact.
	\item Once we work in the eightfold $\Mp(2n, k)$, there is a canonically defined element in the fiber over $-1 \in \Sp(2n, k)$, which we denote by the same symbol $-1$. Decompose the Weil representation $\omega_\psi$ into even and odd pieces, namely $\omega^+_\psi \oplus \omega^-_\psi$. Then $-1 \in \Mp(2n,k)$ is characterized by 
		\[ \omega_\psi^\pm(-1) = \pm {\rm id}. \]
		Consequently $(-1)^2 = 1 $ in $\Mp(2n,k)$. Furthermore, the center of $\Mp(2n,k)$ is generated by $-1$ and $\mu_8$.
	\item For any parabolic subgroup $P \subset \Sp(2n)$ with unipotent radical $U$, there exists a $P(k)$-equivariant section $U(k) \to \Mp(2n,k)$ of the metaplectic covering. In fact, this holds for any covering group by \cite[Appendix I]{mw}. This allows us to define parabolic induction and Jacquet functors.
	\item Recall that the Levi subgroups of $\Sp(2n,k)$ are of the form $M = \prod_i \GL(n_i, k) \times \Sp(2m, k)$ where $\sum_i n_i + m = n$. In the eightfold covering $\Mp(2n,k)$, the preimage $\tilde{M}$ of $M(k)$ splits into
		\[ \tilde{M} = \prod_i \GL(n_i, k) \times \Mp(2m, k) \]
		which follows from an inspection of the Schr\"{o}dinger models. For detailed characterizations of these splittings, we refer to \cite{wwli1,wwli2}.
\end{itemize}
Upon recalling the philosophy of cusp forms, we see an obvious inductive structure in the study of $\Mp(2n, k)$ that is particularly suited to the trace formula. There is a routine procedure to pass between the eightfold covering and the twofold one, but we will not delve into this approach here.

 \vskip 5pt

\subsection{\bf Endoscopy.} 
As explained in \S 2, the genuine representation theory of $\Mp(2n,k)$ is intimately connected with the representation theory of $\SO(V_n)$, where $V_n$ stands for a quadratic $k$-space of dimension $2n+1$ and discriminant $1$. We have defined in \S 2 the dual group of $\Mp(2n,k)$ to be that of $\SO(V_n)$, namely $\Sp(2n, \C)$ with trivial Galois action.
In this manner, we obtain the notion of $L$-parameters $\phi: WD_k \to \Sp(2n,\C)$, $A$-parameters $WD_k \times \SL(2,\C) \to \Sp(2n,\C)$, and the centralizer groups $C_\phi = Z_{\Sp(2n,\C)}(\mathrm{Im}(\phi))$, etc. 

The theory of endoscopy for $\Mp(2n,k)$ is developed in \cite{wwli1}, based on earlier works of Adams \cite{adams} and Renard \cite{renard} in the real case. We summarize this theory as follows.
\begin{enumerate}
	\item The elliptic endoscopic data are in bijection with pairs $(n',n'')$ of non-negative integers such that $n'+n''=n$, up to equivalence. The non-elliptic endoscopic data are defined as elliptic endoscopic data of Levi subgroups, the $\GL$-components playing no role here.
	\item The correspondence of semisimple conjugacy classes and the transfer factor $\Delta(\gamma, \tilde{\delta})$ are defined for every endoscopic datum. The transfer of orbital integrals and the fundamental lemma for units have also been established in \cite{wwli1}. As pointed out by Caihua Luo, we take this opportunity to indicate an obvious mistake in \cite[D\'{e}finition 5.9]{wwli1}: the $a''$ there should be the parameter for $\delta''$, not for $\gamma''$.
\end{enumerate}

Two questions arise:
\vskip 5pt

\noindent (1) In what aspect does $\Mp(2n,k)$ ``look like'' $\SO(2n+1,k)$? 
\vskip 5pt

\noindent  (2) What are the key differences between $\Mp(2n,k)$ and $\SO(2n+1,k)$? In other words, what is the ``metaplecticness'' of $\Mp(2n,k)$? 

\vskip 5pt

Let us turn to $\SO(2n+1)$ first.

\emph{Caution}: hereafter, we use $\pi$ to denote genuine representations of $\Mp(2n,k)$, and use $\tau$ to denote representations of $\SO(2n+1,k)$.

\subsection{Review of Vogan packets for $\SO(2n+1)$}
Set $H = \SO(2n+1)$, a split semisimple $k$-group with Langlands dual group $H^\vee = \Sp(2n, \C)$. For simplicity we only consider the tempered dual $\Irr_{\rm temp}(H)$ of $H(k)$. The results below are due to Arthur \cite{Art1} in their complete generality.

The local Langlands correspondence gives a decomposition
\[ \Irr_{\rm temp}(H) = \bigsqcup_{\phi \in \Phi_{\rm bdd}(H)} \Pi_\phi \]
where $\Phi_{\mathrm{bdd}}(H)$ denotes the set of equivalence classes of tempered $L$-parameters of $H$, and $\Pi_\phi$ are the corresponding $L$-packets. To explain the notation, note that such $\phi$ are required to have bounded image in $\Sp(2n,\C)$.

For any complex group $A$, we write $\pi_1(A, 1)$ for the group of connected components of $A$. The internal structure of each $\Pi_\phi$ is controlled by the groups
\begin{align*}
	C_\phi & := Z_{H^\vee}(\text{Im}(\phi)), \\
	\bar{C}_\phi & := C_\phi/ Z(H^\vee), \\
	\bar{S}_\phi & := \pi_0(\bar{C}_\phi, 1).
\end{align*}
Also write $C_{\phi, \mathrm{ss}}$, $\bar{C}_{\phi, \mathrm{ss}}$ to denote the semi-simple loci of these complex affine groups. Observe that $\bar{S}_\phi$ is a finite product of $\{\pm 1\}$. The local Langlands correspondence comes with an injection
\begin{align*}
	\Pi_\phi & \hookrightarrow \Irr(\bar{S}_\phi) \\
	\tau & \mapsto \langle \cdot, \tau \rangle
\end{align*}
that maps generic $\tau$ to the trivial character, and is bijective for non-archimedean $k$.

Up to equivalence, the elliptic endoscopic data of $H$ are in bijection with pairs of non-negative integers $(n',n'')$ with $n'+n''=n$, identifying $(n',n'')$ and $(n'',n')$; the corresponding endoscopic group is $H^! := \SO(2n'+1) \times \SO(2n''+1)$. The transfer of orbital integrals, fundamental lemma, etc. are nowadays well-known for $H$.
\begin{rem}
	Although $H$ and $\Mp(2n,k)$ share the same dual group $\Sp(2n, \C)$, the elliptic endoscopic data of $H$ admit more symmetries. This can be summarized by the following principle: when working with $\Mp(2n,k)$, one should disregard the symmetries coming from $\pm 1 \in \Sp(2n,\C)$.
\end{rem}

Given a tempered $L$-parameter $\phi: WD_k \to H^\vee$ and $s \in \bar{C}_{\phi, \mathrm{ss}}$, we have
\begin{itemize}
	\item an endoscopic datum (not always elliptic) determined by $s$, whose endoscopic group $H^!$ satisfies $(H^!)^\vee = Z_{H^\vee}(s)$;
	\item a factorization of $L$-parameter
		\[ \phi = WD_k \xrightarrow{\phi^!} (H^!)^\vee \hookrightarrow \check{H}, \quad \text{with $\phi^!$ tempered}; \]
	\item the Whittaker-normalized transfer $f \leadsto f^!$ of test functions; here we fix Haar measures of $H(k)$ and $H^!(k)$.
\end{itemize}

\begin{defn}
	Let $f \in C^\infty_c(H(k))$ and $\tau$ be an admissible representation of finite length of $H(k)$. We follow Arthur's notation to define
	\[ f(\tau) := \mathrm{tr}\left( \tau(f) \right). \]
	The same notation pertains to $H^!(k)$ and $\Mp(2n,k)$, and so forth.
\end{defn}

The local Langlands correspondence (LLC) holds for $H^!$. Define
\[ f^!(\phi, s) = f^!(\phi^!) = \sum_{\tau \in \Pi_{\phi^!}} f^!(\tau). \]
This is the stable character attached to $\phi^!$ evaluated against $f^!$; when $s=1$ one obtains the stable character associated to the packet $\Pi_\phi$. We are ready to state the endoscopic character relation for $H$.
\begin{thm}
	For all  $(\phi, s)$ as above,
	\[ f^!(\phi, s) = \sum_{\tau \in \Pi_\phi} \langle s, \tau \rangle f(\tau). \]
	In particular, $f^!(\phi,s)$ depends only on the image of $s$ in $\bar{S}_\phi$.
\end{thm}

Now turn to \emph{Vogan packets}. Define $S_\phi := \pi_0(C_\phi, 1)$, which is still isomorphic to a direct sum of copies of $\pm 1$. Then one can define larger packets $\Pi_\phi$ to obtain a decomposition
\[ \bigsqcup_{V_n} \Irr_{\rm temp}(\SO(V_n)) = \bigsqcup_{\phi \in \Phi_{\rm bdd}(H)} \Pi_\phi \]
where $V_n$ ranges over the $2n+1$-dimensional quadratic $k$-spaces of discriminant $1$, such that for each $\phi$ there is now a \emph{bijection}
\begin{align*}
	\Pi_\phi & \xrightarrow{1:1} \Irr(S_\phi) \\
	\tau & \longmapsto \langle \cdot, \tau \rangle
\end{align*}
that maps generic $\tau$ to the trivial character. We refer to \cite[\S\S 9--10]{ggp} for a concise introduction.
\vskip 5pt

Summing up, we obtain a simple, uniform LLC at the cost of considering all $\SO(V_n)$ (equivalently, all pure inner forms of $H$) at once. In view of the recipe via the theta correspondence $\theta_\psi$ described in Theorem \ref{T:2.1}, the Vogan packets of $H := \SO(2n+1)$ are what one should match with the required $L$-packets of $\tilde{G} := \Mp(2n,k)$. This is also compatible with our earlier philosophy that the symmetries from $\{\pm 1\} = Z(H^\vee)$ should be disregarded for $\Mp(2n,k)$, which amounts to replacing $\bar{S}_\phi$ by $S_\phi$. We want to understand these phenomena via endoscopy, in terms of transfer of orbital integrals and character relations for $\Mp(2n,k)$.
\vskip 10pt

\subsection{Desiderata for LLC of $\Mp(2n,k)$}
In this subsection, we describe some desiderata of the LLC for $\Mp(2n,k)$, especially with regards to the local endoscopic character identities.
\vskip 5pt

Put $G := \Sp(2n)$ and $\tilde{G} := \Mp(2n,k) \twoheadrightarrow G(k)$. To begin with, we describe the tempered $L$-parameters $\phi$ for $\tilde{G} := \Mp(2n,k)$ in terms of linear algebra. The set of equivalence classes of tempered $L$-parameters will be denoted by $\Phi_{\rm bdd}(\tilde{G})$; it is the same as $\Phi_{\rm bdd}(H)$, where $H := \SO(2n+1)$ as usual.

Regard a tempered $L$-parameter $\phi$ as a symplectic representation of $WD_k$ with underlying space $V_\phi$ of dimension $2n$. As $\phi$ has bounded image, one can decompose $(\phi, V_\phi)$ into a sum of irreducible unitarizable representations
\[ \phi = \sum_{i \in I^+_\phi} \ell_i \phi_i \boxplus \sum_{i \in I^-_\phi} \ell_i \phi_i \boxplus \sum_{i \in J_\phi} \ell_i(\phi_i \boxplus \check{\phi}_i) ,\]
where $\ell_i$ are non-negative integers, $\phi_i \mapsto \check{\phi}_i$ is the contragredient, and $I^\pm_\phi$, $J_\phi$ are indexing sets for irreducible representations such that
\begin{itemize}
	\item $I^+_\phi$ consists of symplectic $\phi_i$;
	\item $I^-_\phi$ consists of orthogonal $\phi_i$, and the corresponding $\ell_i$ are all even;
	\item $J_\phi$ consists of orbits $\{\phi_i, \check{\phi}_i \}$ with $\phi_i \not\simeq \check{\phi}_i$.
\end{itemize}
Hence
\begin{align*}
	C_\phi & = \prod_{i \in I^+_\phi} \OO(\ell_i, \C) \times \prod_{i \in I^-_\phi} \Sp(\ell_i, \C) \times \prod_{i \in J_\phi} \GL(n_i, \C), \\
	S_\phi & = \{\pm 1\}^{I^+_\phi}.
\end{align*}
The notation is the same as the case for $H := \SO(2n+1)$ as $\tilde{G}$ and $H$ share the same dual group. The tempered $L$-parameter $\phi$ is discrete if and only $I^-_\phi = J_\phi = \varnothing$ and $\ell_i \leq 1$ for all $i$. The discrete parameters are expected to depict the genuine discrete series of $\tilde{G}$. Also notice that $-1 \in \Sp(2n,\C)$ always centralizes $\phi$ and corresponds to $(-1, \cdots, -1) \in C_\phi$ under the identification above.

As explained in \S 1, the same recipe also works in the global case.

Given an elliptic endoscopic datum $(n',n'')$ for $\tilde{G}$, denote by $G^! = \SO(2n'+1) \times \SO(2n''+1)$ the corresponding endoscopic group. For any $C^\infty_c$ test function $f$ on $\tilde{G}$ that is \emph{anti-genuine}, i.e. $f(z\tilde{x})=z^{-1}f(\tilde{x})$ for all $z \in \ker(\tilde{G} \to G(k)) \subset \C^\times$, we have a transfer $f^! \in C^\infty_c(G^!(k))$ defined by matching orbital integrals. The function $f^!$ is not unique, but its stable orbital integrals and its values on stable characters are determined by $f$. The same constructions generalize to non-elliptic endoscopic data by passing to Levi subgroups; see \cite[\S\S 3.3--3.4]{wwli4} for details.

Denote by $\widetilde{\Irr}_{\rm temp}(\tilde{G})$ the tempered genuine dual of $\tilde{G}$. By \cite[\S 6.3]{wwli4}, given $\phi \in \Phi_{\rm bdd}(\tilde{G})$ we have
\[ \left[ f^!(\phi, s) := f^!(\phi^!) \right] = \sum_{\pi \in \widetilde{\Irr}_{\rm temp}(\tilde{G})} \Delta(\phi^!, \pi) f(\pi), \quad s \in C_{\phi, \mathrm{ss}} \]
where
\begin{itemize}
	\item $f$ is an anti-genuine $C^\infty_c$ test function on $\tilde{G}$;
	\item the pair $(\phi, s)$ determines an endoscopic datum for $\tilde{G}$ (not always elliptic), with endoscopic group $G^!$, together with a tempered $L$-parameter $\phi^!$ for $G^!$, cf. the case for $\SO(2n+1)$;
	\item $f \leadsto f^!$ is the transfer for the aforementioned endoscopic datum.
\end{itemize}
The right-hand side is a virtual character with coefficients $\Delta(\phi^!, \pi)$ to be determined.

\begin{hypo}\label{hyp:Mp-LLC-approx}
	In order to proceed, we need a first approximation of the tempered local Langlands correspondence for $\tilde{G}$.
	\begin{itemize}
		\item Assume that $\widetilde{\Irr}_{\rm temp}(\tilde{G}) = \bigsqcup_{\phi} \Pi_\phi$, where $\phi$ ranges over $\Phi_{\rm bdd}(\tilde{G})$ and $\Pi_\phi$ are nonempty finite sets, called the $L$-packets for $\tilde{G}$.
		\item Furthermore, assume that the genuine discrete series are depicted by discrete $\phi$.
		\item For all pairs $(\phi, s)$ above and all $\pi \in \widetilde{\Irr}_{\rm temp}(\tilde{G})$, we have
			\[ \pi \in \Pi_\phi \iff \Delta(\phi^!, \pi) \neq 0; \]
			in other words, the virtual character $f \mapsto f^!(\phi, s)$ involves only the members from $\Pi_\phi$, and the $L$-packets are determined by the character relation.
	\end{itemize}
	Under these postulates, the nonzero coefficients in the character relation can be rewritten as
	\[ \Delta(\phi^!, \pi) = \langle s, \pi \rangle, \]
	where $\pi \in \Pi_\phi$. One might be tempted to impose all the usual properties of $\langle s, \pi \rangle$ to the metaplectic case. Some of them turn out to be false, and we shall only impose the following minimal requirements.
	\begin{itemize}
		\item When $\phi \in \Phi_{\mathrm{bdd}}(\tilde{G})$, we expect $\langle 1, \pi \rangle \in \{0,1\}$, i.e. $f^!(\phi, 1)$ yields the ``stable character'' on $\tilde{G}$.
		\item $\langle s, \pi \rangle$ depends only on the conjugacy class of $s$ in $C_{\phi, \mathrm{ss}}$.
	\end{itemize}
\end{hypo}

By dualizing the transfer of test functions, we can lift stable distributions on $G^!(k)$ to invariant genuine distributions on $\tilde{G}$. Unwinding definitions, one sees that the stable character of $G^!(k)$ indexed by $\phi^!$ lifts to the genuine invariant distribution $f \mapsto f^!(\phi^!)$, which is a virtual character with coefficients $\Delta(\phi^!, \pi) = \langle s, \pi \rangle$.

Translation by $-1 \in Z(\tilde{G})$ induces an automorphism of the space of invariant genuine distributions. Here comes a key property.
\begin{thm}
	Given an elliptic endoscopic datum $(n',n'')$ for $\tilde{G}$, denote by $\mathcal{T}_{(n',n'')}$ the dual of transfer. Then
	\[ (\text{translation by } -1) \circ \mathcal{T}_{(n',n'')} = \mathcal{T}_{(n'',n')} \circ \mathrm{swap}_* \]
	where $\mathrm{swap}: \SO(2n'+1) \times \SO(2n''+1) \to \SO(2n''+1) \times \SO(2n'+1)$ is the obvious isomorphism.
\end{thm}

\begin{cor}\label{cor:s-vs-minus}
	For all $s \in C_{\phi, \mathrm{ss}}$ such that $s^2=1$, we have $\langle -s, \pi \rangle = \omega_\pi(-1) \langle s, \pi \rangle $ for all $\pi \in \widetilde{\Irr}_{\rm temp}(\tilde{G})$. Here $\omega_\pi$ stands for the central character of $\pi$.
\end{cor}
\begin{proof}
	The condition $s^2=1$ ensures that the corresponding endoscopic datum is elliptic. Replacing $s$ by $-s$ amounts to swapping $n'$ and $n''$; the two components of $\phi^!$ are also swapped.  Therefore the assertion follows from the Theorem.
\end{proof}
Call the factor $\omega_\pi(-1) \in \{\pm 1\}$ the \emph{central sign} of $\pi$.

Before embarking on the study of $\langle s, \pi \rangle$, we shall define an analogue via $\theta$-lifting.
\begin{defn}
	Suppose $\phi \in \Phi_{\rm bdd}(\tilde{G})$, $\pi \in \widetilde{\Irr}_{\rm temp}(\tilde{G})$ and $\pi = \theta_\psi(\tau)$ for some $\tau \in \Irr_{\rm temp}(\SO(V_n))$, via the recipe from \S 2. Define
	\[ \langle s, \pi \rangle_\Theta := \langle s, \tau \rangle, \quad s \in C_{\phi, \mathrm{ss}} \]
	where the right-hand side comes from the description of Vogan packets for $\SO(2n+1)$. Note that $\phi$ is also a tempered $L$-parameter for $\SO(V_n)$ and the group $C_\phi$ is the same.
\end{defn}

Furthermore, we have the equality for doubling $\epsilon$-factors defined by Lapid and Rallis
\[ \epsilon\left(s, \pi \times \chi, \psi\right) = \epsilon\left(s, \tau \times \chi, \psi\right) \]
for any continuous character $\chi: k^\times \to \C^\times$ by \cite{gs}. We define $\epsilon(s, \phi \times \chi, \psi)$ to be $\epsilon(s, \tau \times \chi, \psi)$ for any $\tau$ in the $L$-packet determined by $\phi: WD_k \to \Sp(2n,\C)$. Note that $\epsilon\left( \frac{1}{2}, \phi, \psi\right)$ is independent of the choice of $\psi$.

\subsection{Case study: $n=1$}
When $n=1$, the works of Waldspurger, Adams, Schultz settled all the postulates in the Hypothesis \ref{hyp:Mp-LLC-approx}; see also \cite{g3}. More precisely, $\widetilde{\Irr}_{\rm temp}(\tilde{G})$ is partitioned into packets of size $1$ or $2$, which is compatible with the recipe via $\theta_\psi$, and we have $\langle 1, \pi \rangle = 1$ for every tempered genuine $\pi$. This also determines $\langle -1, \pi \rangle$ as follows.
\begin{itemize}
	\item Suppose $\pi$ comes from $\SO(V)$ by $\theta$-lifting, $\dim V=3$ with discriminant $1$, then $\langle -1, \pi \rangle_\Theta \in \{\pm 1\}$ equals the Hasse invariant of $V$.
	\item The Corollary \ref{cor:s-vs-minus} says $\langle -1, \pi\rangle = \omega_\pi(-1)$, which equals $\langle -1, \pi \rangle_\Theta \epsilon\left( \frac{1}{2}, \phi, \psi\right)$ by \cite[Theorem 1.4]{gs}.
\end{itemize}

Suppose that there exists a non-selfdual parameter $\phi_0$ such that $\phi = \phi_0 \boxplus \check{\phi}_0$. We get
\[ C_\phi = \GL(1, \C), \quad \langle \cdot, \pi \rangle_\Theta = 1 \]
and $\langle -1, \pi \rangle = \phi_0 \circ \mathrm{rec}_k(-1)$ is not always trivial; here $\mathrm{rec}_k$ is Artin's reciprocity homomorphism for $k$. A closer look reveals that
\begin{enumerate}
	\item $\langle \cdot, \pi\rangle$ does not factor through $S_\phi$ (since $C_\phi$ is connected);
	\item in general, $\langle \cdot, \pi\rangle: C_\phi \to \C^\times$ is not a homomorphism (one can check that $\langle s, \pi \rangle = 1$ whenever $s \neq -1$);
	\item $\langle s, \pi\rangle =  \epsilon\left( \frac{1}{2}, V_\phi^{s=-1}, \psi \right) \langle s, \pi\rangle_\Theta$ for all $s \in C_{\phi, \mathrm{ss}}$.
\end{enumerate}
Here $V_\phi^{s=-1}$ is the $(-1)$-eigenspace under $s$, which is still a representation of $WD_k$. This is in clear contrast with the case of reductive groups.

\subsection{The local conjecture}
Observations in the case $n=1$ together with the multiplicity formula in the global case give some support for the following.
\begin{conj}\label{conj:coeff}
	There should be a local Langlands correspondence $\widetilde{\Irr}_{\rm temp}(\tilde{G}) = \bigsqcup_{\phi \in \Phi_{\mathrm{bdd}}(\tilde{G})} \Pi_\phi$ with character relations that satisfies the Hypothesis \ref{hyp:Mp-LLC-approx}. The assignment $\pi \mapsto \langle \cdot, \pi \rangle: C_{\phi, \mathrm{ss}} \to \C^\times$ should yield
	\begin{gather*}
		\Pi_\phi \xrightarrow{1:1} \Irr(S_\phi) \cdot \tilde{\epsilon}_\phi \quad (\text{a torsor under } \Irr(S_\phi) )
	\end{gather*}
	where
	\begin{gather*}
		\tilde{\epsilon}_\phi(s) := \epsilon\left( \frac{1}{2}, V_\phi^{s=-1} , \psi \right).
	\end{gather*}
	The factor $\tilde{\epsilon}_\phi(s)$ can be expressed more explicitly if we fix a decomposition of $\phi$ into irreducibles, as done in \S 1.
\end{conj}

Note that
\begin{enumerate}
	\item $\tilde{\epsilon}_\phi$ does not belong to $\Irr(S_\phi)$ although their global product does;
	\item when $\phi$ is a discrete-series parameter for $\tilde{G}$, we can show $\tilde{\epsilon}_\phi \in \Irr(S_\phi)$, but this is somehow misleading;
	\item it is possible to formulate a version for $A$-packets, but we prefer to stay in the tempered setting here.
\end{enumerate}

The previous conjecture can be refined as follows.
\begin{conj}\label{conj:coeff-refinement}
	We expect that
	\[ \langle s, \pi \rangle = \langle s, \pi\rangle_\Theta \tilde{\epsilon}_\phi(s), \quad s \in C_{\phi, \mathrm{ss}} \]
	for all $\phi \in \Phi_{\mathrm{bdd}}(\tilde{G})$ and $\pi \in \Pi_\phi$.
\end{conj}

 \begin{rem}
	As a reality check, suppose $s \in C_{\phi, \mathrm{ss}}$ and $s^2 = 1$, so that $s$ determines an elliptic endoscopic datum $(n', n'')$ together with an $L$-parameter $\phi^! = \phi' \times \phi''$ for $\SO(2n'+1) \times \SO(2n''+1)$; we have $\phi = \phi' \boxplus \phi''$. If $s$ is replaced by $-s$, the factor $\langle s, \pi\rangle_\Theta$ changes by $\epsilon(V_n)$, the Hasse invariant of $V_n$, if $\pi$ comes from $\tau \in \Irr(\SO(V_n))$. On the other hand, to $-s$ are attached to endoscopic datum $(n'',n')$ and the swapped parameter $\phi'' \times \phi'$. Hence the factor $\langle s, \pi \rangle = \Delta(\phi^!, \pi)$ also gets swapped. This might appear perplexing in ``the most symmetric case'' in which $n$ is even, $n'= \frac{n}{2} = n''$ and $\phi' = \phi_0 = \phi''$ for some $L$-parameter $\phi_0$ for $\Mp(n,k)$. Let us check Conjecture \ref{conj:coeff-refinement} in this situation.

	In that case, $-s$ is conjugate to $s$ inside $C_\phi$. Indeed, $C_\phi$ is the product of groups of the form
	\[ \OO(2\ell, \C), \quad \Sp(2\ell, \C), \quad \GL(2\ell, \C) \]
	corresponding to summands in $\phi_0$ of the form $\ell \xi$ (for symplectic, orthogonal $\xi$), or $\ell(\xi \boxplus \check{\xi})$ (for non-selfdual $\xi$) respectively. Furthermore, the components of $s$ therein are all conjugate to diagonal matrices
	\[ s^\flat = \text{diag}(\underbrace{1, \ldots, 1}_{\ell \text{ terms}}, \underbrace{-1, \ldots, -1}_{\ell \text{ terms}}). \]
	It is routine to verify that $s^\flat$ is conjugate to $-s^\flat$ in all the three cases. Hence $\langle s, \pi \rangle = \langle -s, \pi \rangle$

	Moreover we have $\langle s, \pi \rangle_\Theta = \langle -s, \pi \rangle_\Theta$; this may be checked directly by noting that $\epsilon(V_n)=1$ whenever $\Irr(\SO(V_n))$ intersects the Vogan packet attached to $\phi$. Hence our conjecture holds only when $\epsilon\left( \frac{1}{2}, V_\phi^{s=1}, \psi \right) = \epsilon\left( \frac{1}{2}, V_\phi^{s=-1}, \psi \right)$; this is indeed true since both terms equal the symplectic root number $\epsilon\left( \frac{1}{2}, \phi_0, \psi \right) \in \{ \pm 1 \}$. We are grateful to Yifeng Liu for alerting us to this issue.
\end{rem}

The second reality check is the rank-one case.
\begin{thm}
	The two conjectures above hold when $n=1$.
\end{thm}
\begin{proof}
	As remarked above, the requirements in Hypothesis \ref{hyp:Mp-LLC-approx} are verified. Furthermore,
	\begin{itemize}
		\item the $L$-packets $\Pi_\phi$ have size $1$ or $2$;
		\item the genuine discrete series are depicted by discrete $\phi$, and the other tempered $\pi$ are parabolically and irreducibly induced from genuine unitary characters of $\tilde{T} \simeq T(k) \times \mu_8$, where $T$ is a split torus in $G$;
		\item the $L$-packets $\Pi_\phi$ are the same as those obtained from $\theta_\psi$.
	\end{itemize}
	It remains to check $\langle s, \pi \rangle = \langle s, \pi \rangle_\Theta \tilde{\epsilon}_\phi(s)$. The case $s=1$ is trivial. We have verified previously that $\langle -1, \pi \rangle = \langle -1, \pi\rangle_\Theta \epsilon\left( \frac{1}{2}, \phi, \psi \right)$, so the case $s=-1$ also holds.

	If $\pi$ is a genuine discrete series, its $L$-parameter $\phi$ is then an irreducible symplectic representation, therefore $C_\phi = \OO(1, \C) = \{\pm 1\}$ and we are done. If $\pi$ is parabolically induced from $\tilde{T}$, then $\langle \cdot, \pi \rangle_\Theta = 1$ and $C_\phi = \SL(2,\C)$ or $\GL(1, \C)$. Suppose $s \neq \pm 1$, then $s \in \Sp(2, \C)$ cannot have eigenvalue $\pm 1$, so $\epsilon\left( \frac{1}{2}, V_\phi^{s=-1}, \psi \right) = 1$. On the other hand, such an $s$ corresponds to the non-elliptic endoscopic datum of $\Mp(2, k)$, namely the identity endoscopic datum for the the split torus $T$. In this case $G^! = T$, $\phi^! = \phi_0$, $\phi = \phi_0 \boxplus \check{\phi}_0$ and $\pi$ is parabolically induced from the genuine character of $\tilde{T}$ arising from $\phi_0$. We are reduced to show $\langle s, \pi\rangle = \Delta(\phi^!, \pi) = 1$, which follows essentially from the \emph{parabolic descent} for endoscopic transfer.
\end{proof}

\vskip 5pt

\subsection{\bf The case of general $n$.}
We describe some recent developments in the case of general $n$. As mentioned in \S 2, especially Theorem \ref{T:2.1} and Corollary \ref{C:2.2}, the theta correspondence provides an LLC for $\Mp(2n, k)$. Moreover, it was shown in \cite{gs} that almost all assertions in Hypothesis \ref{hyp:Mp-LLC-approx} hold, except for the third bullet-point there. 
In a recent preprint \cite{luo2} of Caihua Luo,   released while this paper is being prepared, Conjectures \ref{conj:coeff} and \ref{conj:coeff-refinement} were shown for the local L-packets defined by Corollary \ref{C:2.2}.
This provides a reconciliation of the endoscopic and the $\theta$-lifting descriptions for $\widetilde{\Irr}_{\rm temp}(\tilde{G})$, but is constrained only to local L-packets: it does not treat the case of nontempered local A-packets. 
The main tool used by Luo is the elliptic stable trace formula of the second author (a topic we shall turn to next) and the multiplicity formula described in Theorem \ref{T:B} and \S 2.4, so that his argument is global in nature.
The case of epipelagic $L$-packets for non-archimedean $k$ of residual characteristic $p \gg 0$ seems accessible by purely local arguments, cf. \cite{Kal15}.

\vskip 5pt

 \section{\bf Stable Trace Formula and Local Intertwining Relations}

In this section, we consider  a number field $F$ with $\A = \A_F = \prod_v F_v$. Choose an additive character $\psi = \prod_v \psi_v: F \backslash \A \to \C^\times$ and consider the global eightfold metaplectic covering $\Mp(2n, \A) \twoheadrightarrow \Sp(2n, \A)$. By using local theta lifts, we may talk about local $L$-packets for $\Mp(2n, F_v)$ at each place $v$. Global elliptic $L$- and $A$-parameters are defined as in the first part, and for each global parameter $\Psi$ one defines the global component group $S_\Psi$.

\subsection{\bf Automorphic discrete spectrum.} We recall some results from \S 2. Each parameter $\Psi$ gives rise to a near equivalence class of genuine automorphic representations. Therefore $\tilde{\mathcal{A}}$ decomposes into $\bigoplus_\Psi \tilde{\mathcal{A}}_{\Psi,\psi}$ accordingly. Suppose henceforth that $\Psi$ is a tempered parameter. At each place $v$ of $F$, one can localize $\Psi$ to a local parameter $\Psi_v: WD_k \to \Sp(2n, \C)$ by the ``seed theorem'' \cite[Theorem 1.4.2]{Art1}. We cannot assure $\Psi_v$ to be tempered at a given $v$, but the following is known:
\begin{itemize}
	\item There exists a parabolic $P$ with Levi $M$ such that $\Psi_v$ comes from a tempered parameter $\Psi_{M,v}$ for $\tilde{M}$, twisted by some $\lambda$ in the open chamber $\mathfrak{a}^+_P$.
	\item By \cite[\S 1.5]{Art1} we have $S_{\Psi_v} = S_{\Psi_{M,v}}$, and the packet $\Pi_{\Psi_v}$ consists of normalized parabolic inductions $\mathcal{I}_{\tilde{P}}(\pi_{M,\lambda})$, where $\pi_M$ ranges over the tempered packet of $\tilde{M}$ attached to $\Psi_{M,v}$.
	\item It is expected that $\mathcal{I}_{\tilde{P}}(\pi_{M,\lambda})$ is irreducible and unitarizable, see \cite[Conjecture 8.3.1]{Art1}.
\end{itemize}

For each place $v$, let $\eta_v \in \Irr(S_{\Psi_v})$  correspond via Corollary \ref{C:2.2} (i.e. local theta lifting) to a genuine irreducible representation $\pi_v \in \Pi_{\Psi_v}$, which means
\[ \langle s, \pi_v \rangle_\Theta = \eta_v(s), \quad s \in S_{\Psi_v}. \]
This makes sense even when $\Psi_v$ is non-tempered by the foregoing discussion. Furthermore, suppose $\pi_v$ is unramified for almost all $v$. Then $\eta := \bigotimes_v \eta_v$ is a well-defined character of $S_{\Psi, \A} := \prod_v S_{\Psi_v}$. There is a diagonal homomorphism $\Delta: S_\Psi \to S_{\Psi, \A}$. Consider the abstract representation $\pi = \bigotimes_v \pi_v$ of $\Mp(2n,\A)$ attached to $\eta$. The multiplicity formula stated in \S 2.4 asserts that the multiplicity of $\pi$ in $\tilde{\mathcal{A}}_{\Psi, \psi}$ equals
\[ m_\eta = \begin{cases}
	1, & \text{if } \Delta^*(\eta) = \tilde{\epsilon}_\Psi \\
	0, & \text{otherwise},
\end{cases} \]
where $\tilde{\epsilon}_\Psi(s) := \epsilon\left( \frac{1}{2}, V_\Psi^{s=-1}, \psi \right)$ (for $s \in C_{\Psi, \mathrm{ss}}$) is defined as in the local case. The global avatar has an agreeable property: $\tilde{\epsilon}_\Psi$ descends to a character of $S_\Psi$.

The multiplicity formula is in clear contrast with the case of $\SO(2n+1)$, in which $\tilde{\epsilon}_\Psi$ is replaced by the trivial character. How to understand this in terms of trace formula? Heuristics for Arthur's multiplicity formula for general reductive groups can be found in \cite[\S 4.8]{Art1}. The point is that the characters $\eta_v$ intervene through their appearance as coefficients in the local character relation. If this is to be done for $\Mp(2n,\A)$, one should work with $\langle \cdot, \pi_v \rangle$ instead of $\eta_v = \langle \cdot, \pi_v \rangle_\Theta$. Conjecture \ref{conj:coeff-refinement} says that they differ by a factor $\tilde{\epsilon}_{\Psi_v}: C_{\Psi_v, \mathrm{ss}} \to \C^\times$ at least when $\pi_v$ is tempered. Their global product $\tilde{\epsilon}_\Psi = \prod_v \tilde{\epsilon}_{\Psi_v}$ is precisely what accounts for the deviation of our multiplicity formula from Arthur's; recall that $\tilde{\epsilon}_\Psi$ takes value in $\pm 1$.
\vskip 5pt

\subsection{\bf Stable Trace Formula.}
A natural strategy to classify the automorphic discrete spectrum of $\Mp(2n)$ is to exploit the (conjectural) stable trace formula for $\Mp(2n)$ and  adapt the \emph{standard model} in \cite[Chapter 4]{Art1} to $\Mp(2n,\A)$. The conjectural stable trace formula for $\Mp(2n)$ should read:
\[  I_{disc}^{\Mp(2n)}   =  \sum_{n'+n'' =n} \iota(n',n'')\cdot  \mathcal{T}_{(n',n'')}  \left( S_{disc}^{H_{n',n''} } \right), \] 
where $\mathcal{T}_{(n', n''})$ is the dual of global transfer from $\Mp(2n)$ to the endoscopic group  $H_{n',n''} = \SO(2n'+1) \times \SO(2n''+1)$, $S_{disc}^{H_{n',n''} }$ is the stable part of  the stable trace formula of $H_{n',n''}$ and 
\[  \iota(n',n'') = \begin{cases} 
1/4, & \text{  if $n'n'' \ne 0$;} \\
1/2, & \text{  otherwise.} \end{cases} \]
If this approach succeeds, one would obtain at every place $v$ the character relations for $\Mp(2n,F_v)$, many properties of representations of $\Mp(2n, F_v)$, and ultimately a multiplicity formula for non-tempered representations in the global setting. Note that many of the local properties have been studied in \cite{gs} using $\theta$-lifts. Arthur's analysis in \cite{Art1} ought to be easier for $\Mp(2n, \A)$ since
\begin{itemize}
	\item the stable side of the trace formula is known;
	\item the global parameters and the localizations are already defined --- they are exactly the parameters for $H$;
	\item one can resort to the multiplicity formula above in some local-global arguments, whenever necessary.
\end{itemize}
Nonetheless, to exploit the trace formula, we still need another ingredient which should be proved together with the multiplicity formula in a long, interlocked induction argument, namely the \emph{local intertwining relations}. This will be the subject of the next few paragraphs.

\subsection{Speculations on local intertwining relations}
The constructions below are modeled upon \cite[Chapter 2]{Art1} and we only give a sketch of the relevant ideas and the new features. Details will appear elsewhere.

Begin with the local setting. Take a local field $k$ of characteristic zero, fix $\psi: k \to \C^\times$ and consider the local metaplectic covering $\tilde{G} = \Mp(2n,k) \twoheadrightarrow G(k) := \Sp(2n,k)$. Fix a $k$-pinning $(B,T, \cdots)$ for $G$. In fact, we choose $(B,T)$ by choosing the standard symplectic basis of the $2n$-dimensional symplectic space defining $\Mp(2n,k)$.

Consider a proper standard parabolic $\tilde{P} = \tilde{M}U$ of $\tilde{G}$ and their dual $\tilde{M}^\vee \subset \tilde{P}^\vee \subset \tilde{G}^\vee$. Set $W(M) = W(\tilde{M}^\vee)$ to be $N_G(M)/M$; it acts algebraically on $M$.

\begin{defn}
	First of all, we lift every $w \in W(M)$ to a representative $\tilde{w} \in \tilde{G}$. The idea is to use the \emph{Springer section} (see \cite{spr}): we begin with the case of minimal $M$ for which $W(M)$ is the usual Weyl group $W^G_0$. For any simple root $\alpha$, as a first approximation we lift $w_\alpha$ to
	\[ x_\alpha(1) x_{-\alpha}(-1) x_\alpha(1) \quad \in \tilde{G} \]
	using the splitting over unipotent radicals, where $x_{\pm \alpha}: \mathbb{G}_a \to U_{\pm \alpha}$ is part of the chosen $k$-pinning. If $\alpha$ is a short root and $n > 1$, this furnishes $\tilde{w}_\alpha \in \tilde{G}$. If $\alpha$ is long, we modify the element above by a canonical factor from $\mu_8$ to get $\tilde{w}_\alpha$ such that $\tilde{w}_\alpha^2 = -1$. To be precise, we may safely work within $\Mp(2,k)$ and choose $\tilde{w}_\alpha$ to be the preimage of
	\[ x_\alpha(1) x_{-\alpha}(-1) x_\alpha(1) = \begin{pmatrix} 1 & 1 \\ 0 & 1 \end{pmatrix} \begin{pmatrix} 1 & 0 \\ -1 & 1 \end{pmatrix} \begin{pmatrix} 1 & 1 \\ 0 & 1 \end{pmatrix} = \begin{pmatrix} 0 & 1 \\ -1 & 0 \end{pmatrix} \]
	that acts in the Schr\"{o}dinger model as the unitary Fourier transform on the Schwartz-Bruhat space $\mathcal{S}(k)$. This is exactly the operator on $\mathcal{S}(k)$ denoted by $M_\ell\left[\bigl( \begin{smallmatrix} 0 & 1 \\ -1 & 0 \end{smallmatrix} \bigr)\right]$ in \cite[2.4.1]{wwli1}.

	In general, we use a reduced decomposition of $w \in W^G_0$ to lift it to $\tilde{G}$. If $M$ is not minimal and $w \in W(M)$, we take the representative of $w$ in $W^G_0$ that stabilizes the simple roots of $(B \cap M, T)$. The upshot is that the non-multiplicativity of $w \mapsto \tilde{w}$ is described by a $2$-cocycle similar to that in \cite{Art1}.
\end{defn}

\begin{prop}
	The definition above does not depend on the reduced decomposition of $w$.
\end{prop}
\begin{proof}
	It suffices to consider the case $W(M) = W^G_0$. As in the case \cite[8.3.3 and 9.3.3]{spr} of reductive groups, the key ingredient is the braid-like relation
	\[ \underbrace{\tilde{w}_\alpha \tilde{w}_\beta \cdots}_{m_{\alpha\beta} \text{ terms}} =  \underbrace{\tilde{w}_\beta \tilde{w}_\alpha \cdots}_{m_{\alpha\beta} \text{ terms}}, \quad (m_{\alpha\beta})_{\alpha,\beta}: \text{the Coxeter matrix} \]
	that holds in $\tilde{G}$ for any simple roots $\alpha \neq \beta$. This issue has been addressed in \cite[Proposition 11.4]{bd} for coverings, but without introducing the modification by $\mu_8$ above. Fortunately, the modification of $\tilde{w}_\alpha$ for long $\alpha$ does not alter this relation, since $m_{\alpha\beta}=4$ when $\alpha \neq \beta$, thus the numbers of appearance of $\tilde{w}_\alpha$ on both sides are the same.
\end{proof}

\begin{rem}
	Let $\alpha$ be a long simple root. Although $\tilde{w}_\alpha$ does not coincide with $x_\alpha(1) x_{-\alpha}(-1) x_\alpha(1)$ in our construction, they have the same global product. Indeed, let $F$ be a number field, there exists a unique splitting $i: \Sp(2n, F) \to \Mp(2n, \mathbb{A}_F)$. For a parabolic subgroup $P$ with unipotent radical $U$, we contend that
	\[ i|_{U(F)} = (\prod_v s_v)\bigg|_{U(F)} \]
	where $s_v: U(F_v) \to \Mp(2n, F_v)$ is the canonical splitting over unipotent radicals. Both sides are $P(F)$-equivariant sections, hence the equality follows from the uniqueness part of \cite[Appendix I, Proposition (a)]{mw}.

	On the other hand, $i$ can be realized by Schr\"{o}dinger models (see \cite[Proposition 2.16]{wwli1}, for example.) Now we may assume $n=1$ and use Schr\"{o}dinger models to deduce that $i \left(\begin{smallmatrix} 0 & 1 \\ -1 & 0 \end{smallmatrix}\right)$ equals the product over all places $v$ of the unitary Fourier transform on $\mathcal{S}(F_v)$, that is, the product of $\tilde{w}_{\alpha, v}$. It follows that the product of $\tilde{w}_{\alpha,v}$ equals 
	\[ i \begin{pmatrix} 0 & 1 \\ -1 & 0 \end{pmatrix} = i \begin{pmatrix} 1 & 1 \\ 0 & 1 \end{pmatrix} i \begin{pmatrix} 1 & 0 \\ -1 & 1 \end{pmatrix} i \begin{pmatrix} 1 & 1 \\ 0 & 1 \end{pmatrix}. \]
	By applying the previous observation to $U = \bigl( \begin{smallmatrix} * & 1 \\ & * \end{smallmatrix} \bigr)$ and $\bigl( \begin{smallmatrix} * & \\ 1 & * \end{smallmatrix} \bigr)$, we see this equals the global version of $x_\alpha(1) x_{-\alpha}(-1) x_\alpha(1)$, as desired.
\end{rem}

Next, assume the local Langlands correspondence, character relations, etc. for $\tilde{M}$; this is actually an assumption on the $\Mp$-component of $\tilde{M} = \prod_i \GL(n_i) \times \Mp(2m, k)$, hence can be assured by induction on $n$. Denote by $\Phi_{2, \mathrm{bdd}}(\tilde{M})$ the set of discrete tempered $L$-parameters of $\tilde{M}$, say taken up to equivalence. Write $A_{\tilde{M}^\vee}$ for the maximal central torus of $\tilde{M}^\vee$.
\begin{itemize}
	\item Take $\phi_M$ to be a discrete tempered $L$-parameter for $\tilde{M}$ and $\pi_M \in \Pi_{\phi_M}$ --- this is a simplifying assumption, and one has to allow general cases.
	\item Let $\phi$ be the composite of $\phi_M$ with $\tilde{M}^\vee \hookrightarrow \tilde{G}^\vee$, which is a tempered parameter for $\tilde{G}$. We have $A_{\tilde{M}^\vee} \subset C_\phi$.
	\item Set $W_\phi$ to be the Weyl group of $A_{\tilde{M}^\vee}$ in $C_\phi$; that is, the group of automorphisms of $A_{\tilde{M}^\vee}$ induced from the adjoint action of $C_\phi$.
\end{itemize}

Let $w \in W_\phi$. Since $w$ is an automorphism of $A_{\tilde{M}^\vee}$, it normalizes $\tilde{M}^\vee$ thus can be viewed as an element of $W(M)$, and $\tilde{w} \in \tilde{G}$ is constructed as above. As in \cite[\S 2.3]{Art1}, the study of the full trace formula requires normalized intertwining operators
\begin{gather*}
	R_P(w, \pi_M, \phi) = \pi(w) \ell(w, \pi_M, \phi) R_{wPw^{-1}|P}(\pi_M, \phi)
\end{gather*}
that act as
\begin{gather*}
	I_{\tilde{P}}(\pi_M) \xrightarrow{R_{wPw^{-1}|P}(\pi_M, \phi)} I_{w\tilde{P}w^{-1}}(\pi_M) \xrightarrow{\ell(w, \pi_M, \phi)} I_{\tilde{P}}(w\pi_M) \xrightarrow{\pi(w)} I_{\tilde{P}}(\pi_M).
\end{gather*}
Their definitions are quite non-trivial. Let us make a quick sketch.
\begin{enumerate}
	\item $w\pi_M(\tilde{x}) = \pi_M(\tilde{w}^{-1} \tilde{x} \tilde{w})$ for all $\tilde{x} \in \tilde{M}$.
	\item $R_{wPw^{-1}|P}(\pi_M, \phi)$ is the normalized intertwining operator \`{a} la Arthur; in view of the induction hypothesis, the normalizing factors from $\mathrm{SO}(2n+1)$ work here.
	\item $\ell(w, \pi_M, \phi)$ is $\phi \mapsto \phi(\tilde{w}^{-1} \cdot)$ times some correction factor, again using that from $\SO(2n+1)$-case.
	\item Choose $\pi(w) \in \mathrm{Isom}_{\tilde{M}}(w\pi_M, \pi_M)$: it affects only the $\GL$-components of $\pi_M$. To get a well-defined operator, one may require $\pi(w)$ to be Whittaker-normalized.
\end{enumerate}
Unlike the case of classical groups, the resulting operator is neither multiplicative in $w$ nor Whittaker-normalized. Denote by $\gamma(w, \phi) \in \mathbb{C}^\times$ the ``expected effect'' of $R_P(w, \pi_M, \phi)$ on the space of Whittaker functionals. See for example \cite{bfh} or \cite[Theorem 4.3]{sz} for explicit formulas. It is essentially a product of local root numbers.

Take any $u \in N_\phi := N_{C_\phi}(A_{\tilde{M}^\vee})$ which is semisimple and maps to $w \in W_\phi$ under the natural homomorphism $N_\phi \to W_\phi$. Let $\mathfrak{N}_\phi := \pi_0(N_\phi, 1)$. By applying \cite[p.104]{Art1} or \cite[Pg. 103, after (2.4.1)]{kmsw} to $\SO(2n+1)$, the natural map $S_{\phi_M} \to \mathfrak{N}_\phi$ admits a canonical splitting
\[ \mathfrak{N}_\phi = S_{\phi_M} \times W_\phi. \]
Let $u \mapsto \underline{u} \in \mathfrak{N}_\phi$, and let $\underline{u}^\flat$ be the $S_{\phi_M}$-component of $\underline{u}$ relative to the splitting above. Choose any preimage $u^\flat \in C_{\phi_M, \mathrm{ss}}$ of $\underline{u}^\flat$. Also denote by $\phi^\flat$ (resp. $\pi_M^\flat$) the $\Mp$-component of the parameter $\phi_M$ (resp. the representation $\pi_M$). We want to define
\begin{align*}
	c(u, u^\flat, \phi) & := \tilde{\epsilon}_\phi(u) \tilde{\epsilon}_{\phi^\flat}(u^\flat)^{-1} \gamma(w,\phi)^{-1}, \\
	f(\phi,u) & := c(u, u^\flat, \phi) \sum_{\pi_M \in \Pi_{\phi_M}} \langle u^\flat, \pi_M^\flat \rangle \mathrm{tr}\left( R_P(w, \pi_M, \phi) I_{\tilde{P}}(w, f) \right)
\end{align*}
where $f$ is an anti-genuine test function on $\tilde{G}$. Although $c(u, u^\flat, \phi)$ depends on the choice of $u^\flat$, we have the following:
\begin{prop}
	The expression $f(\phi, u)$ is independent of the choice of $u^\flat$.
\end{prop}
\begin{proof}
	Let $u^\flat, v^\flat \in C_{\phi_M, \mathrm{ss}}$ be any two preimages of $\underline{u}^\flat$. It suffices to show that
	\[ \frac{ \langle u^\flat, \pi_M^\flat \rangle }{ \langle v^\flat, \pi_M^\flat \rangle } = \frac{ \tilde{\epsilon}_{\phi^\flat}(u^\flat) }{ \tilde{\epsilon}_{\phi^\flat}(v^\flat) }. \]
	This equality follows from the Conjecture \ref{conj:coeff} for $\tilde{M}$ since $u^\flat$, $v^\flat$ have the same image in $S_{\phi_M}$.
\end{proof}

Recall that we have defined $f^!(\phi, u)$ for all anti-genuine $C^\infty_c$ test function $f$ on $\tilde{G}$.
\begin{conj}[Special case of local intertwining relation]\label{conj:LIR}
	Given $\phi_M$ and $u$ as above, we expect that $f^!(\phi, u) = f(\phi, u)$ for all $f$.
\end{conj}
Granting the local Langlands correspondence and character relations for proper Levi subgroups, this will give the coefficients $\langle \cdot, \pi \rangle$ for non-discrete tempered $L$-packets for $\tilde{G}$ and much more, such as information on Knapp--Stein $R$-groups. We refer to \cite[Chapter 2]{Art1} for a full explanation.

\begin{rem}
	Here are some heuristics behind the Conjecture \ref{conj:LIR}. Suppose that
	\begin{itemize}
		\item $\tau_M$ is a generic discrete series of a Levi subgroup $M_\SO = \prod_i \GL(n_i) \times \SO(2m+1)$ of $\SO(2n+1)$;
		\item $\pi_M \in \widetilde{\Irr}(\tilde{G})$ is obtained from $\tau_M$ by applying $\theta$-lift to the $\SO$-component of $\tau_M$, which is $\psi$-generic (see \cite[\S 9]{gs} for relevant notions);
		\item take $\tau \in \Irr(\SO(2n+1))$ to be the generic summand of the normalized parabolic induction of $\tau_M$, and take $\pi$ to be the $\theta$-lift of $\tau$.
	\end{itemize}
	Take $\phi_M$ to be the $L$-parameter of $\pi_M$. Again, by the results furnished by \cite{gs}, we know $\pi$ is a $\psi$-generic summand of the normalized parabolic induction of $\pi_M$. Granting the Conjecture \ref{conj:coeff-refinement}, the coefficient of $f(\pi)$ in $f^!(\phi, u)$ should be
	\[ \langle u, \pi\rangle =  \langle u, \pi\rangle_\Theta \tilde{\epsilon}_\phi(u) = \langle u, \tau\rangle \tilde{\epsilon}_\phi(u) = \tilde{\epsilon}_\phi(u) \]
	by the Whittaker-normalization for $\SO(2n+1)$. On the other hand, the operator $\gamma(w, \phi)^{-1} R_P(w, \pi_M, \phi)$ acts trivially on the space of $\psi$-Whittaker functionals of $\mathcal{I}_{\tilde{P}}(\pi_M)$. Therefore the coefficient of $f(\pi)$ in $f(\phi, u)$ equals
	\[ \tilde{\epsilon}_\phi(u) \cdot \frac{\langle u^\flat, \pi_M \rangle}{ \tilde{\epsilon}_{\phi^\flat}(u^\flat) } = \tilde{\epsilon}_\phi(u) \cdot \langle u^\flat, \pi_M \rangle_\Theta \]
	which reduces to $\tilde{\epsilon}_\phi(u) \langle u^\flat, \tau_M \rangle = \tilde{\epsilon}_\phi(u)$, as what one can expect from $f^!(\phi, u) = f(\phi, u)$.
\end{rem}

The local intertwining relation is subject to global constraints.
\begin{enumerate}
	\item First of all, we have to allow non-discrete $\phi_M$, or even $A$-parameters for global applications.
	\item Suppose $\phi_M$ is a global parameter and $u, u^\flat$ are chosen globally. The constants should satisfy a product formula $\prod_v c(u, u^\flat, \phi)_v = 1$.
	\item The analogues of the \emph{sign lemmas} in \cite[Chapter 4]{Art1} must hold. They are the Lemmas 4.3.1 and 4.4.1 of \textit{loc. cit.}, which are related to the operators $R_P(w, \pi_M, \phi)$ and constructions on the dual groups, respectively. This is what motivates our definition of $R_P(w, \pi_M, \phi)$.
	\item Ultimately, we want to feed these objects into the \emph{standard model} of \textit{loc. cit.}, and deduce all the theorems inductively.
\end{enumerate}

To achieve all these goals, we will need the fundamental lemma for the spherical Hecke algebra of $\tilde{G}$, and this has recently been shown by
 Caihua Luo \cite{luo1}.  Moreover, we will probably need the stabilization of the full trace formula for $\tilde{G}$. Thus far, only the elliptic terms are stabilized: this is the main result of \cite{wwli3}.

\end{document}